\documentclass{amsart}
\usepackage{graphicx}
\usepackage{amssymb}
\usepackage{amscd}
\usepackage{longtable}
\usepackage{pb-diagram}

\newtheorem{theorem}{Theorem}[section]
\newtheorem{proposition}[theorem]{Proposition}
\newtheorem{lemma}[theorem]{Lemma}

\newtheorem{problem}[theorem]{Problem}

\newtheorem{corollary}[theorem]{Corollary}
\newtheorem*{claim*}{Claim}
\theoremstyle{definition}

\newtheorem{example}[theorem]{Example}

\newcommand{\G}{G(K)}
\newcommand{\C}{\Bbb C}
\newcommand{\Z}{\Bbb Z}
\newcommand{\Q}{\Bbb Q}
\newcommand{\D}{\Delta}

\newcommand{\la}{\langle}
\newcommand{\ra}{\rangle}

\newcommand{\ord}{\mathcal{O}}

\newcommand{\ind}{\mathrm{Ind}}

\theoremstyle{remark}
\newtheorem{remark}[theorem]{Remark}

\numberwithin{equation}{section}
\begin{document}
\title{Twisted Alexander vanishing order of knots}

\author{Katsumi Ishikawa, Takayuki Morifuji, and Masaaki Suzuki}



\begin{abstract}
Based on a vanishing theorem for non-fibered knots due to Friedl and Vidussi, we define the twisted Alexander vanishing order of a knot to be the order of the smallest finite group such that the corresponding twisted Alexander polynomial is zero. In this paper, we show its basic properties, and provide several explicit values for knots with $10$ or fewer crossings. Moreover, we characterize a finite group admitting the zero-twisted Alexander polynomial. 
\end{abstract}

\address{Research Institute for Mathematical Sciences, Kyoto University, Kyoto 606-8502, Japan}
\email{katsumi@kurims.kyoto-u.ac.jp}

\address{Department of Mathematics, Hiyoshi Campus, Keio University, Yokohama 223-8521, Japan}
\email{morifuji@keio.jp}

\address{Department of Frontier Media Science, 
Meiji University, 4-21-1 Nakano, Nakano-ku, Tokyo 
164-8525, Japan}
\email{mackysuzuki@meiji.ac.jp}

\maketitle

%%%%%%%%%%%%%%%%%%%%%%%%%%%%%%%%%%%%%%%%%%%%%%%%%%%%%%%%%%%%%%%%%%%%%%%%%%%%%%%%%%%%%%%%%%%%%%%%%%%%%%%%%%%%%%%%%%%%
\section{Introduction}

The Alexander polynomial $\D_K(t)$ of a knot $K$ is a fundamental and 
important tool in the study of knots. It is practical but not sufficient 
to distinguish between two knots. In fact, 
many knots share the same Alexander polynomial, and further, 
the polynomial might be trivial in the sense that $\D_K(t)=1$ holds. 
Nevertheless, it is known that 
the Alexander polynomial never vanishes because it has the following 
basic property: $\D_K(1)=\pm1$. 
Then, it seems a natural question whether the \textit{twisted} Alexander polynomial of a knot associated to a representation has the same property or not. 

The twisted Alexander polynomial is introduced by Lin~\cite{Lin01-1} for knots in $S^3$, and by Wada~\cite{Wada94-1} for finitely presentable groups. 
It is defined for the pair of a knot and its representation, 
and has lots of applications to knot theory and low-dimensional 
topology (see \cite{FV10-1}, \cite{Morifuji15-1}). 
For example, a fibered knot, 
whose complement admits a structure of a surface 
bundle over the circle such that the closures of the fibers are Seifert surfaces, is detected by twisted Alexander polynomials 
associated to regular representations of finite groups 
(see \cite{FV11-1}). 
Furthermore, as a stronger result, 
Friedl and Vidussi show the following vanishing theorem 
for the twisted Alexander polynomial of a non-fibered knot. 

\begin{theorem}[{\cite[Theorem~1.2]{FV13-1}}]\label{thm:FV}
For a non-fibered knot $K$, there exists an epimorphism $f$ of the knot group $G(K)$ onto a finite group $G$ such that the twisted Alexander polynomial associated to the composition of $f$ and the regular representation of $G$ vanishes. 
\end{theorem}

In this paper, we call a finite group $G$ a \textit{twisted Alexander vanishing (TAV) group of a knot $K$} if there exists an epimorphism of $G(K)$ onto $G$ such that the twisted Alexander polynomial associated to the regular representation of $G$ is zero. Thus, in view of Theorem \ref{thm:FV}, 
it is natural to raise the following challenge; 
\textit{for each non-fibered knot $K$, 
find an epimorphism of $G(K)$ onto a TAV group $G$.} 
Moreover, we may ask the following question; 
\textit{what is the smallest TAV group?}

In our previous paper \cite{MS22-1}, 
we exhibited some concrete examples which answer the above questions. In order to state the results of this paper precisely, 
let us introduce the notion of the \textit{twisted Alexander vanishing order} 
$\ord(K)$ of a non-fibered knot $K$. 
Namely, we define $\ord(K)$
to be the order of the smallest TAV group of $K$, 
and call it the \textit{TAV order} of $K$ in short (this number is called the \textit{minimal order} of $K$ in \cite{MS22-1}). 
On the other hand, for a fibered knot $K$, 
we set $\ord(K)=+\infty$,  
because its twisted Alexander polynomial is monic 
(see \cite{Cha03-1}, \cite{FK06-1}, \cite{GKM05-1}), 
and hence never vanishes. Hereafter, by abuse of terminology, 
we also simply call a finite group $G$ a \textit{TAV group} if $G$ is a TAV group of some knot $K$. 

In general, the determination of the TAV order of a given knot seems to be difficult. However, we can show the following basic properties of 
$\ord(K)$. 
Let $\mathcal{K}$ be the set of isotopy classes of oriented knots in the $3$-sphere, and $\mathcal{N}\subset\mathcal{K}$ the subset consisting of non-fibered knots. Then, the TAV order of a knot induces a function 
$\ord \colon \mathcal{K}\to\mathbb{N}\cup\{+\infty\}$. 
Our first theorem of this paper is the following. 

\begin{theorem}\label{thm:main-1}
The TAV order $\ord\colon\mathcal{K}\to\mathbb{N}\cup\{+\infty\}$ has the following properties: 
\begin{itemize}
\item[(i)]
For any knot $K\in\mathcal{K}$, $\ord(K)\geq24$ holds. 
\item[(ii)]
The restriction $\ord|_{\mathcal{N}}\colon\mathcal{N}\to\mathbb{N}$ is unbounded. 
\item[(iii)]
For the connected sum $K_1\#K_2$ of two knots $K_1,K_2\in\mathcal{K}$, it holds that $\ord(K_1\#K_2)=\min\{\ord(K_1),\ord(K_2)\}$.
\item[(iv)]
If there is an epimorphism from $G(K_1)$ to $G(K_2)$, $\ord(K_1)\leq\ord(K_2)$ holds.
\item[(v)]
For a periodic knot $K$ and its quotient knot $K'$, 
$\ord(K)\leq\ord(K')$ holds. 
\item[(vi)]
If there is a proper degree one map $E_K\to E_{K'}$, 
where $E_K$ denotes the exterior of a knot $K$, 
then 
$\ord(K)\leq\ord(K')$ holds. 
\end{itemize}
\end{theorem}

The lower bound in Theorem~\ref{thm:main-1}(i) is the best possible in the sense that there exists a knot $K$ 
whose TAV order $\ord(K)$ attains the bound. More precisely, we can 
determine the TAV order of several knots by computer-aided calculation. We adopt Rolfsen's table \cite{Rolfsen76-1} to represent a prime knot with $10$ or fewer crossings. 
Our second theorem is the following. 

\begin{theorem}\label{thm:main-2}
For any prime knot $K$ with $10$ or fewer crossings, we have 
\begin{itemize}
\item[(i)]
$\ord(K)=24$, if $K=9_{35}, 9_{46}$,
\item[(ii)]
$\ord(K)=60$, if $K=10_{67}, 10_{120}, 10_{146}$,
\item[(iii)]
$\ord(K)=96$, if $K=10_{166}$,
\item[(iv)]
$\ord(K)=120$, if $K=8_{15}, 9_{25}, 9_{39}, 9_{41}, 9_{49}, 10_{58}$,
\item[(v)]
$\ord(K)\geq126$, otherwise.
\end{itemize}
\end{theorem}

Using Theorems \ref{thm:main-1} and \ref{thm:main-2}, we have the following corollary. 

\begin{corollary}\label{cor:main-3}
There are infinitely many knots $K\in\mathcal{N}$ with $\ord(K)=24,60, 96$, or $120$. In particular, there are infinitely many prime knots $K\in\mathcal{N}$ with $\ord(K)=24$. 
\end{corollary}

In the latter half of the paper, 
we characterize a TAV group among the finite groups. 
To prove Theorem~\ref{thm:main-1}(i), we first provide a sufficient condition that the twisted Alexander polynomial is non-zero (see Proposition~\ref{pro:metabelian} for detail), using an isomorphism theorem due to Friedl and Powell~\cite{FP12-1}, which is useful to define an obstruction theory of knot concordance. 
Further, as a much stronger result, we provide a necessary and sufficient condition for a finite group to be a TAV group. 
Namely, we have the following characterization. 

\begin{theorem}\label{thm:main-4}
A finite group $G$ is a TAV group if and only if 
$G$ is normally generated by a single element and 
the commutator subgroup of $G$ is not a $p$-group. 
\end{theorem}

The above theorem follows from a more general statement which relates the vanishing of the twisted Alexander polynomial to a lifting problem of certain group homomorphisms (see Theorem~\ref{lifting-thm} for detail). To the best of our knowledge, these results provide new characterizations of the vanishing of the twisted Alexander polynomial. We remark that the assertion of Proposition~\ref{pro:metabelian} is contained in Theorem~\ref{thm:main-4}, but it is sufficient to prove Theorem~\ref{thm:main-1}(i), and moreover, their proofs are different in nature. 

This paper is organized as follows. 
In Section 2, 
we quickly recall some necessary material for our purpose, 
namely, 
the definition of twisted Alexander polynomials, 
basic properties of representations of finite groups, 
an isomorphism theorem due to Friedl and Powell, 
and group homomorphisms of the knot group of a cable knot in the $3$-sphere. 
In Section 3, 
we prove Theorems~\ref{thm:main-1} and \ref{thm:main-2}, and  Corollary~\ref{cor:main-3}. 
In Section 4, we consider the lifting criteria of a homomorphism from the fundamental group of a connected finite CW complex to a product group, and prove Theorem~\ref{thm:main-4}. 
In Section 5, we provide the upper bound of the TAV 
order of several knots explicitly by computer-aided calculation, 
and discuss related problems on the TAV order and 
twisted Alexander polynomials associated to representations of finite groups.

%%%%%%%%%%%%%%%%%%%%%%%%%%%%%%%%%%%%%%%%%%%%%%%%%%%%%%%%%%%%%%%%%%%%%%%%%%%%%%%%%%%%%%%%%%%%%%%%%%%%%%%%%%%%%%%%%%%%
\section{Preliminaries}
%%%%%%%%%%%%%%%%%%%%%%%%%%%%%%%%%%%%%%%
\subsection{Twisted Alexander polynomials}\label{sub:tap}

Let $X$ be a connected finite CW complex, 
$\phi\in H^1(X;\Z)=\mathrm{Hom}(\pi_1(X),\Z)$, 
and $\rho\colon\pi_1(X)\to \mathrm{GL}(n, R)$ a homomorphism to a general linear group over a Noetherian unique factorization domain $R$. 
Define a right $\Z[\pi_1(X)]$-module structure on 
$R^n\otimes_\Z\Z[t^{\pm1}]=R[t^{\pm1}]^n$ as follows: 
$$
(v\otimes p)\cdot g=(v\cdot \rho(g))\otimes(p\cdot t^{\phi(g)}),
$$
where 
$g\in \pi_1(X)$ and $v\otimes p\in R^n\otimes_\Z\Z[t^{\pm1}]$. 
Here, we view $R^n$ as row vectors. 
Taking tensor product, we obtain a homomorphism 
$\rho\otimes\phi\colon\pi_1(X)\to \mathrm{GL}(n,R[t^{\pm1}])$. 

We denote by $\tilde{X}$ the universal covering of $X$, 
and use the homomorphism 
$\rho\otimes\phi$ to regard 
$R[t^{\pm1}]^n$ as a right $\Z[\pi_1(X)]$-module. 
The chain complex $C_*(\tilde{X})$ is a left 
$\Z[\pi_1(X)]$-module via deck transformations. 
We can therefore consider the tensor products
$$
C_*(X;R[t^{\pm1}]^n)
:=R[t^{\pm1}]^n\otimes_{\Z[\pi_1(X)]}C_*(\tilde{X}),
$$ 
which form a chain complex of $R[t^{\pm1}]$-modules. 
We then consider the $R[t^{\pm1}]$-modules 
$H_*(X;R[t^{\pm1}]^n)
:=H_*(C_*(X;R[t^{\pm1}]^n))$. 

Since $X$ is compact and $R[t^{\pm1}]$ is Noetherian, 
these modules are finitely presented over $R[t^{\pm1}]$. 
We define the \textit{twisted Alexander polynomial} of $(X,\phi,\rho)$ 
to be the order of $H_1(X;R[t^{\pm1}]^n)$ 
as a left $R[t^{\pm1}]$-module. 
We will denote it as $\D_{X,\phi}^\rho(t)\in R[t^{\pm1}]$, 
and note that $\D_{X,\phi}^\rho(t)$ is well defined up to multiplication by a unit in $R[t^{\pm1}]$. See \cite{FV11-1} for other basic properties of twisted Alexander polynomials.

%%%%%%%%%%%%%%%%%%%%%%%%%%%%%%%%%%%%%%%
\subsection{Representations of finite groups}
We quickly recall basic properties of representations of finite groups according to 
Fulton-Harris~\cite{FH}. 

A \textit{representation}\, of a finite group $G$
 on a finite-dimensional complex vector space $V$ is a homomorphism $\rho \colon G\to \mathrm{GL}(V)$ 
 of $G$ to the group of automorphisms of $V$. Such a map gives $V$ the structure of a $G$-module. 
 We also call $V$ itself a representation of $G$. 
 A representation $V$ is called \textit{irreducible} 
 if there is no proper non-zero invariant 
 subspace of $V$. 
 
In this paper,  
we will need the following well known properties of representations of finite groups.

\begin{lemma}\label{lem:abelian}
Every irreducible representation $V$ of a finite abelian group $G$ is one-dimensional.
\end{lemma}

For a finite group $G$, let $V$ be a complex vector space of dimension $|G|$, where $|G|$ denotes the order of $G$, with a basis 
$\{e_g\,|\, g\in G\}$. For $h\in G$, let $\rho(h)$ be the linear map of $V$ into $V$ which sends $e_g$ to $e_{gh}$; this defines a linear representation, 
which is called the (right) \textit{regular representation} of $G$.

%\begin{lemma}\label{lem:regular-rep}
%Any irreducible representation $V$ of a finite group $G$ appears as a direct summand in the regular representation $\dim\,V$ times.
%\end{lemma}

%Namely, 
\begin{lemma}\label{lem:regular-rep}
The regular representation $\rho\colon G \to \mathrm{GL}(V)$ 
%\mathrm{GL}(|G|,\Z) \subset \mathrm{GL}(|G|,\C)$ 
is equivalent to a direct sum $\bigoplus_{i=1}^k\rho_i^{\oplus \dim V_i}$ 
where $\rho_i\colon G\to \mathrm{GL}(V_i)~(1\leq i\leq k)$ are 
the irreducible representations of $G$. 
\end{lemma}

\begin{lemma}\label{lem:FH3.14}
Let $H$ be a subgroup of a finite group $G$. Then, 
the regular representation of $G$ is induced from the regular representation of $H$. 
\end{lemma}

Throughout this paper, let us consider a knot $K$ in the $3$-sphere $S^3$ 
and let $E_K=S^3\setminus \nu(K)$, where $\nu(K)$ denotes an open tubular neighborhood of $K$. 
We denote $\pi_1(E_K)$ by $G(K)$, 
and call it the \textit{knot group} of $K$. 
If $f\colon\G\to G$ is an epimorphism to a finite group $G$, 
then we get the representation 
$$
G(K)\overset{f}{\longrightarrow} G
\overset{\rho}{\longrightarrow} \mathrm{Aut}_\Z(\Z[G]),
$$ 
where the second map is given by the right multiplication. 
We can also identify $\mathrm{Aut}_\Z(\Z[G])$ with 
$\mathrm{GL}(|G|,\Z)$, and 
obtain the corresponding twisted Alexander polynomial 
$\D_{E_K,\phi}^{\rho\circ f}(t)$. 
For the \textit{abelianization} homomorphism 
$\phi \colon G(K)\to H_1(E_K;\Z)\cong\Z$, 
we drop $\phi$ from the notation and use $\D_K^{\rho\circ f}(t)$ for simplicity. 

\begin{remark}\label{rmk:torsion}
It is known that 
$\D_{K}^{\rho\circ f}(t)\not=0$ if and only if 
$H_1(E_K;\Q[G][t^{\pm1}])=H_1(E_K;\Q[t^{\pm1}]^{|G|})$ is $\Q[t^{\pm1}]$-torsion, 
namely, 
$\mathrm{rank}_\Z\,H_1(E_K;\Z[t^{\pm1}]^{|G|})$ is finite 
(see Turaev~\cite[Remark~4.5]{Turaev01-1}). 
\end{remark}

The following proposition is very useful for our purpose (see \cite[Proposition~2.5]{MS22-1}). 

\begin{proposition}\label{pro:brunching-rule}
For the (regular) representation $\rho\circ f\colon\G\to G\to \mathrm{GL}(|G|,\Z) \subset \mathrm{GL}(|G|,\C)$, 
$$
\D_{K}^{\rho\circ f}(t)
=\prod_{i=1}^k\left(\D_{K}^{\rho_i\circ f}(t)\right)^{\dim\,\rho_i}
$$
holds, 
where each $\rho_i$ is the irreducible representation of $G$ appeared in Lemma \ref{lem:regular-rep}. 
\end{proposition}

\begin{example}\label{ex:cyclic}
Let $G$ be a cyclic group of order $n$. 
The regular representation $\rho \colon G\to\mathrm{GL}(n,\Z) \subset \mathrm{GL}(n,\C)$ can be decomposed 
into one-dimensional irreducible representations $\rho_1, \rho_2, \ldots, \rho_n$, 
where $\rho_j \colon G\to \mathrm{GL}(1,\C)$ is determined by $\rho_j(g) = \alpha^j$ 
for a generator $g$ of $G$ and a primitive $n$-th root $\alpha\in\C$ of unity.  
It is easy to see that $\D_K^{\rho_j\circ f}(t) = \D_K(\alpha^j t)$.
%, where $\alpha\in\C$ is a primitive $n$-th root of unity. 
Then for an epimorphism $f \colon G(K)\to G$, we have $\D_K^{\rho\circ f}(t)=\prod_{j=1}^{n}\D_K(\alpha^j t)$.
In particular, $\D_K^{\rho\circ f}(t)\not=0$. 
\end{example}

Using Proposition~\ref{pro:brunching-rule}, 
if we can find an irreducible representation $\rho_i$ of $G$ such that 
$\D_{K}^{\rho_i\circ f}(t)=0$, then 
we have an explicit example of a TAV group.

%%%%%%%%%%%%%%%%%%%%%%%%%%%%%%%%%%%%%%%
\subsection{Non-vanishing of twisted Alexander polynomials}

The following isomorphism lemma 
%theorem 
is the special case of ~\cite[Proposition~4.1]{FP12-1}. We denote a cyclic group $\Z/k\Z$ by $\Z_k$. Given a prime number $p$, a finite group $G$ is a \textit{$p$-group} if and only if the order $|G|$ is a power of $p$. 

\begin{lemma}\label{lem:FP4.1}
Let $p$ be a prime number. 
Suppose that $S,Y$ are finite CW-complexes such that 
there is a map $i\colon S\to Y$ which induces an isomorphism 
$i_*\colon H_*(S;\Z_p){\rightarrow} $ $H_*(Y;\Z_p)$, which, for example, is always the case if $i$ induces a $\Z$-homology equivalence. Let $\phi\colon\pi_1(Y)\to \Z=\la t\ra$ and $\varphi\colon\Z\to \Z_k$ be 
epimorphisms, and $Y_\varphi$ the induced covering of $Y$. 
Let $\phi'\colon\pi_1(Y_\varphi)\to \Z$ be the restriction of $\phi$, 
and let $\rho'\colon\pi_1(Y_\varphi)\to \mathrm{GL}(d,\Q)$ be 
a $d$-dimensional representation, such that $\rho'$ restricted to 
the kernel of $\phi'$ factors through a $p$-group. 
Define $S_\varphi$ to be the pull-back covering $S_\varphi:=i^*(Y_\varphi)$. 
Then 
$$
i_*\colon H_*(S_\varphi;\Q(t)^d)\overset{}{\rightarrow}
H_*(Y_\varphi;\Q(t)^d)
$$
is an isomorphism.
\end{lemma}

Here we refer to \cite[Section 2.1]{FK06-1} for the twisted homology group of a disconnected space. 

Using Lemma~\ref{lem:FP4.1}, we can provide a sufficient condition that the twisted Alexander polynomial is nonzero. 
Before that, we review some basic facts about homomorphisms of the knot group onto a finite group (see \cite{BZH14-1}). It is known that the knot group $G(K)$ can be written as a semi-direct product $\Z\ltimes [G(K),G(K)]$ where $\Z$ is generated by an element $t$, which is a meridian of the knot. It is also known that for a group $G$ there are a knot $K$ and an epimorphism $f \colon \G\to G$ if and only if 
$G$ is finitely generated and $G=\la g^G\ra$, that is, normally generated by a single element $g\in G$ 
(see \cite{GA75-1}, \cite{Johnson80-1}). 
For example, there is no epimorphism of $\G$ onto the dihedral group $D_{2n}=\Z_2\ltimes\Z_{2n}$. 

A finite group $H$ is called \textit{metabelian} 
%($2$-step metabelian) 
if $[H,H]$ is abelian. If $H$ is a metabelian factor group of $G(K)$, then $H$ can be written as $\Z_k\ltimes[H,H]$, and $[H,H]$ is a factor module of the first homology group of the $k$-fold cyclic branched covering of $K$. In particular, a meridian of $K$ is mapped to a generator of $\Z_k$.

%Using Lemma~\ref{lem:FP4.1}, we can provide a sufficient condition that the twisted Alexander polynomial is non-zero. Before that, we review some basic facts about homomorphisms of the knot group onto a finite group. It is known that the knot group $G(K)$ can be written as a semi-direct product $\Z\ltimes [G(K),G(K)]$ where $\Z$ is generated by an element $t$, which is a meridian of the knot. Any cyclic group $\Z$ or $\Z_k$ will be thought of as having a distinguished generator, $t$. We remark that if $G(K)$ maps onto some group $H=\Z_k\ltimes [H,H]$, then a given meridian is mapped to a generator $t^*$ of some cyclic subgroup $\Z_k^*$ such that $H=\Z_k^*\ltimes [H,H]$. Therefore, when representations of a knot group $G(K)$ onto a semi-direct product $\Z_k\ltimes [H,H]$ are mentioned in this paper, it will always be supposed implicitly that a meridian is mapped to the generator $t$ of $\Z_k$. 

\begin{proposition}\label{pro:metabelian}
Let $H$ be a finite metabelian group normally generated by a single element such that $[H,H]$ is a $p$-group. Then, for any epimorphism $f \colon G(K)\to H$ and the regular representation 
$\rho\colon H\to\mathrm{Aut}_{\Q}(\Q[H])$, we have $\D_K^{\rho\circ f}(t)\not=0$. Namely, $H$ is never a TAV group. 
\end{proposition}

\begin{proof}
We may assume $H=\Z_k\ltimes[H,H]$. 
Let $Y=E_K$, $\phi \colon G(K)=\pi_1(Y)\to\Z=\la t\ra$ the abelianization, 
and $Y_\varphi$ the $k$-fold cyclic covering of $Y$ induced by 
$\varphi\colon \Z\to\Z_k$. Let $d=|[H,H]|=p^r$. 
We will show that $H_1(Y;\Q[t^{\pm1}]^{kd})$ is torsion 
over $\Q[t^{\pm1}]$ (see Remark~\ref{rmk:torsion}). 

We first note that $(\rho\circ f)|_{\ker\phi}$ factors through a 
$p$-group $[H,H]$, because $\ker\phi$ is the commutator subgroup of $\pi_1(Y)$. 
Since 
$\pi_1(Y_\varphi)$ also projects to $[H,H]\cong\{1\}\ltimes [H,H]$, we obtain a commutative diagram (with inclusions for the 
horizontal maps):  
\begin{equation*}
\begin{CD}
\pi_1(Y_\varphi)@>>>\pi_1(Y)\\
@V{f'}VV
@VV{f}V\\
[H,H]@>>>\Z_k\ltimes [H,H]
\end{CD}
\end{equation*}
Then, we can consider the representation 
$$
\pi_1(Y_\varphi)\overset{f'}{\longrightarrow}[H,H]
\overset{\rho'}{\longrightarrow}\mathrm{GL}(d,\Q)
$$
where $\rho'$ is the regular representation of $[H,H]$. 
Let $i \colon S^1\to Y$ be the inclusion which represents a meridian 
of the knot $K$, namely, 
this $S^1$ is sent to $1$ under the map $\phi$. 
Since 
$i_* \colon H_*(S^1;\Z_p)\overset{}{\rightarrow}H_*(Y;\Z_p)$ is an 
isomorphism, 
Lemma~\ref{lem:FP4.1} implies that 
$$
i_* \colon H_*(S^1_{\varphi};\Q(t)^d)
\overset{}{\rightarrow}H_*(Y_\varphi;\Q(t)^d)
$$
is an isomorphism, 
where $S^1_{\varphi}$ is the pull-back covering 
$S^1_{\varphi}:=i^*(Y_\varphi)$. 
An elementary calculation shows 
$H_*(S^1_{\varphi};\Q(t)^d)=0$, 
which then implies that 
$H_*(Y_\varphi;\Q(t)^d)\cong0$. 

Next, using $f$ and $f'$, 
we can consider the right action of $\Z[\pi_1(Y)]$ on 
$$
\ind^H_{[H,H]}(\Q^d)=
\Q^d\otimes_{\Z[[H,H]]}\Z[H],
$$ 
and then, obtain a representation
$$
\pi_1(Y)\overset{f}{\longrightarrow}H
\overset{\rho}{\longrightarrow}
\mathrm{GL}(kd,\Q)
$$
where $\rho$ is the regular representation of $H$ 
(see Lemma~\ref{lem:FH3.14}). 
Moreover, 
$(\rho\circ f)\otimes\phi$ induces a representation 
$\pi_1(Y)\to \mathrm{GL}(kd,\Q(t))$. Then, 
we have the following claim 
(we can apply the same proof as in \cite[Claim]{FP12-1} 
to the chain complex $C_*(Y_\varphi;\Q(t)^d)$).

\begin{claim*}\label{lem:shapiro}
$H_*(Y_\varphi;\Q(t)^d)=H_*(Y;\Q(t)^{kd})$.
\end{claim*}

Hence, it follows that 
$H_*(Y;\Q(t)^{kd})\cong0$. 
Since the quotient field $\Q(t)$ is flat over $\Q[t^{\pm1}]$, 
$H_*(Y;\Q[t^{\pm1}]^{kd})$ is torsion over $\Q[t^{\pm1}]$. 
In particular, 
$H_1(Y;\Q[t^{\pm1}]^{kd})$ is $\Q[t^{\pm1}]$-torsion. 
This completes the proof of 
Proposition~\ref{pro:metabelian}. 
\end{proof}

\begin{example}\label{ex:non-vanishing}
The alternating group $A_4=\Z_3\ltimes\Z_2^2$ where $[A_4,A_4]\cong\Z_2^2$, 
the dihedral group $D_{p^n}=\Z_2\ltimes\Z_{p^n}$ where $[D_{p^n},D_{p^n}]\cong\Z_{p^n}$, 
and the metacyclic group $G=G(m,p|k)\cong\Z_m\ltimes\Z_p$, 
where $[G,G]\cong\Z_p$ and $k\in\Z$ is a primitive $m$-th root of $1$ modulo $p$ 
(see Fox~\cite{Fox70-1} for $m=p-1$, Hirasawa-Murasugi~\cite{HM09-1} for $2$-bridge knots, and 
Boden-Friedl~\cite{BF14-1} in general case), 
satisfy the assumption of Proposition~\ref{pro:metabelian}. 
Thus, for these groups, we have $\D_{K}^{\rho\circ f}(t)\not=0$. 
\end{example}

\begin{corollary}\label{cor:order2p}
Let $p$ be an odd prime and $G$ a group of order $2p$. 
If there exists an epimorphism $f \colon G(K)\to G$, then 
$\D_K^{\rho\circ f}(t)\not=0$. 
\end{corollary}

\begin{proof}
It is known that the group of order $2p$ is a cyclic or a dihedral group. 
Hence, the assertion follows from 
Examples~\ref{ex:cyclic} and~\ref{ex:non-vanishing}. 
\end{proof}

%%%%%%%%%%%%%%%%%%%%%%%%%%%%%%%%%%%%%%%
\subsection{Cable knots}
To prove Theorem~\ref{thm:main-1}(ii), we recall the cabling operation of a knot in $S^3$, and provide a sufficient condition that the knot group of a cable knot surjects to a cyclic group. 

A \textit{cable knot} is a satellite knot with pattern knot being a torus knot. In particular, when the companion knot is a knot $K$ and the pattern knot is the $(p,q)$-torus knot, the cable knot is called the \textit{$(p, q)$-cable} of $K$; see Figure \ref{cable-fig}.

\begin{figure}[h]
\centering
\includegraphics[width=8cm]{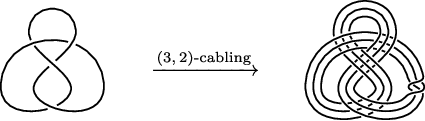}
\caption{The $(3,2)$-cable of the figure-eight knot}\label{cable-fig}
\end{figure}

\begin{proposition}\label{cable-prop}
Let $p$ be a positive integer and $q$ an integer coprime to $p$. Let $K^{(p, q)}$ denote the $(p, q)$-cable of a knot $K$. If the order of a finite group $G$ divides $p$, the image of any group homomorphism $f\colon G(K^{(p, q)}) \to G$ is cyclic.
\end{proposition}

\begin{proof}
Let $N_K$ be a closed tubular neighborhood of $K$ and construct $K^{(p, q)}$ in $\partial N_K$; we regard a $q/p$-slope of $K$ as $K^{(p, q)}$. We take an open tubular neighborhood $\nu(K^{(p,q)})$ of $K^{(p,q)}$ so that the exterior $E_{K^{(p,q)}} = S^3 \setminus
%\backslash 
\nu(K^{(p,q)})$ satisfies 
\begin{itemize}
\item[(i)] 
$E_K \cap E_{K^{(p,q)}} \cong E_K$; 
\item[(ii)] 
$N_K \cap E_{K^{(p,q)}} \cong N_K$; and 
\item[(iii)] 
$\partial N_K \cap E_{K^{(p,q)}}$ is an annulus, whose central curve is a $q/p$-slope of $K$. 
\end{itemize}
Then, we apply the Seifert-Van Kampen theorem to $E_{K^{(p, q)}} = (E_K \cap E_{K^{(p,q)}}) \cup (N_K \cap E_{K^{(p, q)}})$ to obtain
$$G(K^{(p, q)}) \cong (G(K) * \mathbb{Z}) / N.$$
Here, $N$ is the minimal normal subgroup of $G(K) * \mathbb{Z}$ containing $\ell^p m^q \ell'^{-p}$, where $(m, \ell) \in G(K)^2$ is a meridian-longitude pair of $K$ and $\ell'$ is a generator of $\mathbb{Z} \cong \pi_1(N_K \cap E_{K^{(p, q)}})$.
\par Since the order $|G|$ of $G$ divides $p$, we have $f(\ell)^p = f(\ell')^p = e$ and hence $f(m)^q = e$. Because $q$ is coprime to $|G|$, $f(m)^q = e$ implies $f(m) = e$. Recalling that $G(K)$ is normally generated by $m$, we find that the restriction of $f$ to $G(K)$ is trivial. Thus, the image of $f$ is generated by $f(\ell')$.
\end{proof}

%%%%%%%%%%%%%%%%%%%%%%%%%%%%%%%%%%%%%%%%%%%%%%%%%%%%%%%%%%%%%%%%%%%%%%%%%%%%%%%%%%%%%%%%%%%%%%%%%%%%%%%%%%%%%%%%%%%%
\section{Proof}

In this section, we prove Theorems~\ref{thm:main-1}, \ref{thm:main-2}, 
and Corollary~\ref{cor:main-3}. 

%%%%%%%%%%%%%%%%%%%%%%%%%%%%%%%%%%%%%%%
\subsection{Proof of Theorem~\ref{thm:main-1}}\label{sub:main-1}
(i) There are $59$ finite groups of order less than $24$, 
and $24$ of them are abelian. 
We see from Lemma~\ref{lem:abelian} that 
every irreducible representation of a finite abelian group 
is one-dimensional. Thus, 
the corresponding twisted Alexander polynomial 
$\D_K^{\rho\circ f}(t)$ is 
described by the Alexander polynomial $\D_K(t)$, 
and $\D_K(t)$ never vanishes for any knot $K$ 
as stated at the beginning of the introduction. 

From the above argument, we only have to show that the 
twisted Alexander polynomials associated to the regular representations of the following $12$ non-abelian groups never vanish for any knot $K$: 
$$
S_3=D_3=\Z_3\rtimes \Z_2,~
D_5=\Z_5\rtimes\Z_2,~
\mathrm{Dic}_3=\Z_3\rtimes \Z_4,~
A_4=\Z_2^2\rtimes\Z_3,
$$
$$
D_7=\Z_7\rtimes\Z_2,~
D_9=\Z_9\rtimes\Z_2,~
\Z_3\times S_3=\Z_3\rtimes\Z_6,~
\Z_3\rtimes S_3=\Z_3^2\rtimes_2\Z_2,
$$
$$
\mathrm{Dic}_5=\Z_5\rtimes_2\Z_4,~
F_5=\Z_5\rtimes\Z_4,~
\Z_7\rtimes\Z_3,~
D_{11}=\Z_{11}\rtimes\Z_2,
$$
where 
$S_n$ is the symmetric group, $\mathrm{Dic}_n$ is the 
dicyclic (binary dihedral) group, and $F_5$ is the Frobenius group. 

All of the above groups have the form $H=\Z_k\ltimes [H,H]$ where $[H,H]$ is 
an abelian $p$-group, so that the assertion follows from Proposition~\ref{pro:metabelian}. 

(ii) 
Let $K$ be a non-fibered knot. For any positive integer $n$, the $(n!, 1)$-cable $K^{(n!, 1)}$ of $K$ is non-fibered by \cite{hms} (see also \cite{BZH14-1}). By Proposition \ref{cable-prop}, the image of any homomorphism of $G(K^{(n!, 1)})$ to a group of order less than or equal to $n$ is cyclic, and then, 
Example~\ref{ex:cyclic} shows that the twisted Alexander polynomial is not zero. Thus, we have $\mathcal{O}(K^{(n!, 1)}) > n$.

(iii) Let $\rho_i : G(K_i) \to \mathrm{GL}(n,\C), (i=1,2)$ be representations such that $\rho_1(a)$ is conjugate with $\rho_2(b)$ for some meridian loops $a\in G(K_1)$ and $b\in G(K_2)$. Then, Cho constructed in \cite{Cho15-1} a \textit{connected sum} $\rho_1\#\rho_2 : G(K_1\# K_2) \to \mathrm{GL}(n,\C)$ using $\rho_1$ and $\rho_2$. The resulting representation is not unique, but it satisfies the following property: Given a representation $\rho : G(K_1\# K_2) \to \mathrm{GL}(n,\C)$, there exist representations $\rho_i : G(K_i) \to \mathrm{GL}(n,\C), (i=1,2)$, which are unique up to conjugation, such that one of the possible connected sums $\rho_1 \# \rho_2$ equals $\rho$. 
Moreover, it is shown that the product formula 
$$
W(K_1\#K_2,\rho)=
W(K_1,\rho_1)\cdot W(K_2,\rho_2)\cdot\det\left(I_n-(\rho\otimes\phi)(a)\right)
$$
holds, where $W(K,\rho)$ denotes Wada's invariant of a knot $K$ and its representation $\rho$ (see \cite{Wada94-1} for precise definition), which is referred to as the twisted Alexander polynomial or twisted Reidemeister torsion, and $I_n$ is the $n\times n$ identity matrix. Since $\D_K^\rho(t)=0$ if and only if $W(K,\rho)=0$ 
(see \cite[Section~3.3.1]{FV10-1}), applying the formula to our setting, the desired equality $\ord(K_1\# K_2)=\min\{\ord(K_1),\ord(K_2)\}$ can be shown as follows: 
Without loss of generality, we may assume $\ord(K_1)\leq\ord(K_2)$. 
Let $\rho_1\colon G(K_1)\to\mathrm{GL}(n,\C)$ be a representation that realizes $\ord(K_1)$. Then, by the product formula, we have $\D_{K_1\#K_2}^\rho(t)=0$, 
and hence the inequality $\ord(K_1\#K_2) \leq \min\{\ord(K_1), \ord(K_2)\}$ holds. Next, let $\rho\colon G(K_1\# K_2)\to \mathrm{GL}(n,\C)$ be a representation that realizes $\ord(K_1\#K_2)$. Then, the product formula implies $\Delta_{K_i}^{\rho_i}(t)=0$ ($i=1$ or $2$) for representations $\rho_i\colon G(K_i)\to \mathrm{GL}(n,\C)$ such that $\rho=\rho_1\# \rho_2$. Thus, we obtain the desired inequality 
$\ord(K_1\# K_2) \geq \min\{\ord(K_1), \ord(K_2)\}$ in both cases.

(iv) 
If $\pi \colon G(K_1)\to G(K_2)$ is an epimorphism such that 
$\phi_1=\phi_2\circ \pi$ for abelianizations, 
and $f_1=f_2\circ \pi$ for $f_2=G(K_2)\to G$, 
we have $\D_{K_1}^{\rho\circ f_1}(t)
=\D_{K_2}^{\rho\circ f_2}(t)\cdot h(t)$ 
for some $h(t)\in\Z[t^{\pm1}]$ 
(see \cite[Theorem~3.1]{KSW05-1}). 
Thus, we see that the following inequality holds: $\ord(K_1)\leq\ord(K_2)$. 

(v) For a periodic knot $K$ and its quotient knot $K'$, 
there is an epimorphism from $G(K)$ to $G(K')$. 
Hence, (iv) implies the assertion. 

(vi) Since a proper degree one map $E_K\to E_{K'}$ induces 
an epimorphism $G(K)\to G(K')$, 
(iv) implies the assertion. 

This completes the proof of Theorem~\ref{thm:main-1}.  

%%%%%%%%%%%%%%%%%%%%%%%%%%%%%%%%%%%%%%%
\subsection{Proof of Theorem~\ref{thm:main-2}}
In our previous paper \cite[Theorem~3.2]{MS22-1}, 
we showed the statements (i), (ii), $61<\ord(K)\leq120$, if 
$K=8_{15}, 9_{25}, 9_{39},9_{41}, 9_{49}, 10_{58}$, and 
$\ord(K)>61$, otherwise. 

(iii) We provide an epimorphism $f \colon G(10_{166})\to S_4\ltimes \Z_2^2$ such 
that $\D_K^{\rho\circ f}(t)=0$. 
Let us take a presentation of $G(10_{166})$ as follows: 
\begin{align*}
G(10_{166}) = 
\langle x_1, x_2, \ldots, x_{10} \, | \, 
& x_4 x_2 \bar{x}_4 \bar{x}_1 , x_9 x_2 \bar{x}_9 \bar{x}_3 , x_6 x_4 \bar{x}_6 \bar{x}_3 , x_8 x_5 \bar{x}_8 \bar{x}_4 , x_2 x_6 \bar{x}_2 \bar{x}_5 , \\
& x_9 x_7 \bar{x}_9 \bar{x}_6 , x_5 x_8 \bar{x}_5 \bar{x}_7 , x_1 x_9 \bar{x}_1 \bar{x}_8 , x_2 x_9 \bar{x}_2 \bar{x}_{10},  x_7 x_1 \bar{x}_7 \bar{x}_{10}
\rangle
\end{align*}
where $\bar{x}$ is the inverse of $x$. 
We also take a presentation of the finite group 
$G=S_4\ltimes\Z_2^2$ as follows: 
\begin{align*}
G=
\la g_1,g_2,\ldots,g_6\,|\,
&
g_1^2=g_2^2=g_3^2=g_4^2=g_5^3=g_6^2=1, 
g_5g_2g_5^{-1}=g_1g_2=g_2g_1,\\
&
g_1g_3=g_3g_1, g_1g_4=g_4g_1, g_5g_1g_5^{-1}=g_6g_1g_6=g_2, g_2g_3=g_3g_2,\\
&
g_2g_4=g_4g_2, g_6g_2g_6=g_1, g_5g_3g_5^{-1}=g_6g_3g_6=g_3g_4=g_4g_3,\\
&
g_5g_4g_5^{-1}=g_3, g_4g_6=g_6g_4, g_6g_5g_6=g_5^{-1}
\ra.
\end{align*}
It is easy to see that the following map $f \colon G(10_{166}) \to G$ is a homomorphism: 
\begin{align*}
& 
f(x_1) = g_2g_4g_5g_6, ~
f(x_2) = g_6g_3, ~
f(x_3) = g_6g_3, ~
f(x_4) = g_1g_6g_3g_5, \\
&
f(x_5) = g_1g_6g_3g_5,  ~
f(x_6) = g_2g_4g_5g_6, ~
f(x_7) = g_6 g_3 g_5, ~
f(x_8) = g_6 g_3 g_5 , \\
&
f(x_9) = g_1 g_2 g_3 g_6, ~
f(x_{10}) = g_1g_2g_3g_6.
\end{align*} 
It might be unclear that the above homomorphism is actually an epimorphism, but we can check it by the following correspondence: 
\begin{align*}
&
f(x_4^3x_7)=g_1, \quad f(x_4x_9x_1x_2)=g_2, \quad f(x_4^2)=g_3,\\
&
f(x_9^2)=g_4, \quad f(x_1x_4x_9x_1)=g_5, \quad f(x_1^3x_2)=g_6.
\end{align*}

Since the group $G$ can be embedded into the symmetric group $S_8$ and hence in $\mathrm{GL}(8,\Z)$ via permutation 
matrices. Let $\tau \colon G\to\mathrm{GL}(8,\Z)$ denote 
this representation. 
Then, 
we can check that 
the twisted Alexander polynomial of $10_{166}$ 
associated to $\tau \circ f$ is zero. 
Hence, 
by Proposition \ref{pro:brunching-rule}, 
we obtain $\D_{10_{166}}^{\rho\circ f}(t) = 0$ for the 
regular representation 
$\rho \colon G\overset{}{\rightarrow}\mathrm{GL}(96,\Z)$. 
On the other hand, with the aid of a computer, 
we can show that $\D_{10_{166}}^{\rho\circ f}(t) \neq 0$ for 
all non-abelian finite groups of order between $61$ and $96$, 
though $G(10_{166})$ admits epimorphisms onto some of them. 
Therefore, we get $\ord(10_{166}) = 96$. 

(iv) 
We can check that for $807$ non-abelian finite groups of order between $61$ and $120$, 
$\D_K^{\rho\circ f}(t)\not=0$ holds by computer-aided calculations. 

(v) 
Similarly, 
we can also check that for the other knots and 
for any non-abelian finite groups of order between $61$ and $120$, 
$\D_K^{\rho\circ f}(t)\not=0$ holds. 
Finally, non-abelian groups of order between $121$ and $125$ normally generated by a single element 
are the dihedral group $D_{61}$ and the dicyclic group 
$\mathrm{Dic}_{31}=\Z_4\ltimes\Z_{31}$, and their commutator subgroups are abelian $p$-groups, 
so that the corresponding twisted Alexander polynomials of any knot 
are nonzero by Proposition~\ref{pro:metabelian}.

This completes the proof of Theorem~\ref{thm:main-2}. 

%%%%%%%%%%%%%%%%%%%%%%%%%%%%%%%%%%%%%%%
\subsection{Proof of Corollary~\ref{cor:main-3}}

Let $K_1$ be the non-fibered knot $9_{35}$ or $9_{46}$. 
If we take the connected sum $K_2=K_1\#K_1$, 
then it is non-fibered, and we obtain 
$$
\ord(K_2)=\min\{\ord(K_1),\ord(K_1)\}=\ord(K_1)=24
$$ 
by Theorem~\ref{thm:main-1}(iii). 
Inductively, we have $\ord(K_n)=24$ for the connected sum $K_n=\#_n K_1$. Other cases are similar. 

As for the latter assertion, we use the satellite knot construction 
(see \cite{CS16-1} for instance). 
Let $L$ be a non-fibered prime knot with $\ord(L)=24$ (e.g. $L=9_{35}, 9_{46}$). Choose an embedded circle $c$ in $S^3\setminus L$ satisfying the following: 
$c$ is unknotted in $S^3$, $c$ does not bound a $2$-disk in $E_L$, and $L\cup c$ is a prime link. Choose a hyperbolic knot $P$, and $P_n$ be the connected sum $\#_nP$. Let $L_n=L(c,P_n)$ be the satellite knot where $P_n$ is the companion and $L$ viewed as a knot in the solid torus $E_c$ is the pattern. 
Then, it is well known that the knot group $G(L_n)$ surjects to $G(L)$. 
Thus, $\ord(L_n)\leq \ord(L)=24$ holds by Theorem \ref{thm:main-1}(iv), and hence, $\ord(L_n)=24$ by Theorem \ref{thm:main-1}(i). 
A standard argument on satellite construction with distinct companion shows that $L_n$ and $L_m$ are not equivalent for any $n\not=m$. The primality of $L_n$ is also shown in \cite[Theorem 4.5]{CS16-1}. Hence, we have an infinitely many prime non-fibered knots $L_n$ with $\ord(L_n)=24$. This completes the proof of Corollary~\ref{cor:main-3}. 

%%%%%%%%%%%%%%%%%%%%%%%%%%%%%%%%%%%%%%%%%%%%%%%%%%%%%%%%%%%%%%%%%%%%%%%%%%%%%%%%%%%%%%%%%%%%%%%%%%%%%%%%%%%%%%%%%%%%
\section{Vanishing of twisted Alexander polynomials}

In this section, we give a characterization of twisted Alexander vanishing groups of knots. 
For reader's convenience we recall the statement: 

\begin{theorem}[Theorem~\ref{thm:main-4}]\label{char-thm}
A finite group $G$ is a TAV group if and only if 
$G$ is normally generated by a single element and 
the commutator subgroup of $G$ is not a $p$-group. 
\end{theorem}

To prove Theorem~\ref{char-thm}, we first show a more general statement on the vanishing of twisted Alexander polynomials.

%%%%%%%%%%%%%%%%%%%%%%%%%%%%%%%%%%%%%%%%%%%%%%%%%%%%%%%
\subsection{Lifting criteria}

Let $X$ be a connected finite CW complex and $G$ a finite group. As seen in Section~\ref{sub:tap}, the twisted Alexander polynomial $\Delta^{\rho \circ f}_{X, \phi}(t) \in \mathbb{Z}[t^{\pm 1}]$ is defined for homomorphisms $\phi \colon \pi_1(X) \to \mathbb{Z}$ and $f \colon \pi_1(X) \to G$, where $\rho \colon G \to {\rm Aut}_{\mathbb{Z}}(\mathbb{Z}[G])$ is given by the right multiplication. We take $\mathbb{Z}$ as the coefficient ring in this section, but, of course, $\Delta^{\rho \circ f}_{X, \phi}(t)$ is equal to the twisted Alexander polynomial associated to the regular representation with complex coefficient up to multiplication by a unit in $\mathbb{C}[t^{\pm 1}]$.

The following theorem states that the vanishing of twisted Alexander polynomials is equivalent to the existence of a certain lift of group homomorphisms. 

\begin{theorem}\label{lifting-thm}
The twisted Alexander polynomial $\Delta^{\rho \circ f}_{X, \phi}(t)$ is zero if and only if there exists a nontrivial lift $\tilde{f} \colon \pi_1(X) \to \mathbb{Z}[G \times \mathbb{Z}] \rtimes (G \times \mathbb{Z})$, where $G \times \mathbb{Z}$ acts on $\mathbb{Z}[G \times \mathbb{Z}]$ by the left multiplication, of the homomorphism $f \times \phi \colon \pi_1(X) \to G \times \mathbb{Z}$, i.e., a group homomorphism $\tilde{f}$ such that $p_{G \times \mathbb{Z}} \circ \tilde{f} = f \times \phi$ and ${\rm Im}\; \tilde{f} \cap (\mathbb{Z}[G \times \mathbb{Z}] \times \{(e, 0)\}) \neq \{(0; e, 0)\}$, where $p_{G \times \mathbb{Z}} \colon \mathbb{Z}[G \times \mathbb{Z}] \rtimes (G \times \mathbb{Z}) \to G \times \mathbb{Z}$ is the projection.
\end{theorem}

\begin{proof}
We shall show that the following are equivalent; in particular, the theorem asserts the equivalence of (i) and (v):
\begin{itemize}
\item[(i)] $\Delta^{\rho \circ f}_{X, \phi}(t) = 0.$
\item[(ii)] ${\rm Hom}_{\mathbb{Z}[t^{\pm 1}]}(H_1(X_{f \times \phi}), \mathbb{Z}[t^{\pm 1}]) \neq 0.$
\item[(iii)] There exists a $\mathbb{Z}[t^{\pm 1}]$-homomorphism from $H_1(X_{f \times \phi}, \tilde{x})$ to $\mathbb{Z}[t^{\pm 1}]$ whose restriction to $H_1(X_{f \times \phi}) \subset H_1(X_{f \times \phi}, \tilde{x})$ is nontrivial.
\item[(iv)] There exists a nontrivial lift $\tilde{f} \colon \pi_1(X) \to M \rtimes (G \times \mathbb{Z})$ of $f \times \phi$.
\item[(v)] There exists a nontrivial lift $\tilde{f} \colon \pi_1(X) \to \mathbb{Z}[G \times \mathbb{Z}] \rtimes (G \times \mathbb{Z})$ of $f \times \phi$.
\end{itemize}
Here, a covering space $X_{f \times \phi}$ of $X$, a zero-dimensional sub-complex $\tilde{x} \subset X_{f \times \phi}$, and a $(G \times \mathbb{Z})$-module $M$ are defined in the detailed proof below.\\[5pt]
{\bf $\text{(i)} \Leftrightarrow \text{(ii)}$.} We fix a base point $x_0 \in X$ and let $p \colon (X_{f \times \phi}, \tilde{x}_0) \to (X, x_0)$ be the covering associated with the right action of $\pi_1(X)$ on $G \times \mathbb{Z}$: Denote $(G \times \mathbb{Z}) \times_{\pi_1(X)} \tilde{X}$ by $X_{f \times \phi}$, where $(\tilde{X}, y_0) \to (X, x_0)$ is the universal covering, and $(e, 0; y_0) \in X_{f \times \phi}$ by $\tilde{x}_0$. We should remark that $X_{f \times \phi}$ is not necessarily connected, and that there is a left action of $G \times \mathbb{Z}$ on $X_{f \times \phi}$. We denote $(e, 1) \in G \times \mathbb{Z}$ and its action by $s$. Defining $tc = s_*c$ for $c \in H_1(X_{f \times \phi})$, we regard the homology group $H_1(X_{f \times \phi})$ as a $\mathbb{Z}[t^{\pm 1}]$-module, which is isomorphic to $H_1(X; \mathbb{Z}[G][t^{\pm 1}])$.
\par We should recall that the order of a finitely generated $R$-module $M$ over a unique factorization domain $R$ is zero if and only if ${\rm Hom}_R(M, R) \neq 0$; the ``if'' part is obvious from the definition of the order, and the ``only-if'' part can be shown by taking a nontrivial element of ${\rm Hom}_R(M, Q(R))$ ($\neq 0$ by, e.g., \cite[Remark 4.5]{Turaev01-1}), where $Q(R)$ denotes the quotient field of $R$, and multiplying it by an appropriate scalar to make the image contained in $R$. Thus, the twisted Alexander polynomial $\Delta^{\rho \circ f}_{X, \phi}(t)$ is zero if and only if ${\rm Hom}_{\mathbb{Z}[t^{\pm 1}]}(H_1(X_{f \times \phi}), \mathbb{Z}[t^{\pm 1}]) \neq 0$.\\[5pt]
%\par Let us set $\tilde{x} = p^{-1}(x_0)$ and see that $${\rm Hom}_{\mathbb{Z}[t^{\pm 1}]}(H_1(X_{f \times \phi}, \tilde{x}), \mathbb{Z}[t^{\pm 1}]) \to {\rm Hom}_{\mathbb{Z}[t^{\pm 1}]}(H_1(X_{f \times \phi}), \mathbb{Z}[t^{\pm 1}])$$ is surjective. By the long exact sequence for the pair $(X_{f \times \phi}, \tilde{x})$, we have an exact sequence
%$$H^0(\tilde{x}; \mathbb{Z}[t^{\pm 1}]) \xrightarrow{\delta} H^1(X_{f \times \phi}, \tilde{x}; \mathbb{Z}[t^{\pm 1}]) \xrightarrow{j} H^1(X_{f \times \phi}; \mathbb{Z}[t^{\pm 1}]) \to 0;$$
%it is sufficient to show that $j$ induces an epimorphism between the kernels of $s^* - t$. Let $\alpha \in H^1(X_{f \times \phi}; \mathbb{Z}[t^{\pm 1}])$ satisfies $(s^* - t)\alpha = 0$. Since $j$ is surjective, there exists $\beta_0 \in H^1(X_{f \times \phi}, \tilde{x}; \mathbb{Z}[t^{\pm 1}])$ such that $j(\beta_0) = \alpha$. We have $j((s^* - t)\beta_0) = (s^* - t)\alpha =0$, and hence there is $\gamma_0 \in H^0(\tilde{x}; \mathbb{Z}[t^{\pm 1}])$ such that $\delta(\gamma_0) = (s^* - t)\beta_0.$ Because $s^* - t$ is surjective in $H^0(\tilde{x}; \mathbb{Z}[t^{\pm 1}])$, we can take $\gamma \in H^0(\tilde{x}; \mathbb{Z}[t^{\pm 1}])$ such that $(s^* - t)\gamma = \gamma_0$. By setting $\beta = \beta_0 - \delta(\gamma),$ we find $j(\beta) = \alpha$ and $(s^* - t)\beta = 0$, as required.
{\bf $\text{(ii)} \Leftrightarrow \text{(iii)}$.} Let us set $\tilde{x} = p^{-1}(x_0)$. The implication $\text{(iii)} \Rightarrow \text{(ii)}$ is trivial. To show $\text{(ii)} \Rightarrow \text{(iii)}$, we claim that
$${\rm Hom}_{\mathbb{Z}[t^{\pm 1}]}(H_1(X_{f \times \phi}, \tilde{x}), \mathbb{Z}[t^{\pm 1}]) \to {\rm Hom}_{\mathbb{Z}[t^{\pm 1}]}(H_1(X_{f \times \phi}), \mathbb{Z}[t^{\pm 1}])$$
is surjective. By the long exact sequence for the pair $(X_{f \times \phi}, \tilde{x})$, we have an exact sequence
$$H^0_{\mathbb{Z}}(\tilde{x}; \mathbb{Z}[t^{\pm 1}]) \xrightarrow{\delta} H^1_{\mathbb{Z}}(X_{f \times \phi}, \tilde{x}; \mathbb{Z}[t^{\pm 1}]) \xrightarrow{j} H^1_{\mathbb{Z}}(X_{f \times \phi}; \mathbb{Z}[t^{\pm 1}]) \to 0,$$
where $H^*_{\mathbb{Z}}$ denotes the untwisted cohomology group; e.g., $H^*_{\mathbb{Z}}(X_{f \times \phi}; \mathbb{Z}[t^{\pm 1}])$ is the cohomology of the cochain complex ${\rm Hom}_{\mathbb{Z}}(C_*(X_{f \times \phi}), \mathbb{Z}[t^{\pm 1}])$. Because
\begin{align*}
H^1_{\mathbb{Z}}(X_{f \times \phi}, \tilde{x}; \mathbb{Z}[t^{\pm 1}]) &\cong {\rm Hom}_{\mathbb{Z}}(H_1(X_{f \times \phi}, \tilde{x}), \mathbb{Z}[t^{\pm 1}]) \quad \text{and}\\
H^1_{\mathbb{Z}}(X_{f \times \phi}; \mathbb{Z}[t^{\pm 1}]) &\cong {\rm Hom}_{\mathbb{Z}}(H_1(X_{f \times \phi}), \mathbb{Z}[t^{\pm 1}]),
\end{align*}
it is sufficient to show that $j$ induces an epimorphism between the kernels of $s^* - t$. Let $\alpha \in H^1_{\mathbb{Z}}(X_{f \times \phi}; \mathbb{Z}[t^{\pm 1}])$ satisfy $(s^* - t)\alpha = 0$. Since $j$ is surjective, there exists $\beta_0 \in H^1_{\mathbb{Z}}(X_{f \times \phi}, \tilde{x}; \mathbb{Z}[t^{\pm 1}])$ such that $j(\beta_0) = \alpha$. We have
$$j((s^* - t)\beta_0) = (s^* - t)\alpha =0,$$
and hence there is $\gamma_0 \in H^0_{\mathbb{Z}}(\tilde{x}; \mathbb{Z}[t^{\pm 1}])$ such that $\delta(\gamma_0) = (s^* - t)\beta_0.$ Because $s^* - t$ is surjective in $H^0_{\mathbb{Z}}(\tilde{x}; \mathbb{Z}[t^{\pm 1}])$, we can take $\gamma \in H^0_{\mathbb{Z}}(\tilde{x}; \mathbb{Z}[t^{\pm 1}])$ such that $(s^* - t)\gamma = \gamma_0$. By setting $\beta = \beta_0 - \delta(\gamma),$ we find $j(\beta) = \alpha$ and $(s^* - t)\beta = 0$, as required.\\[5pt]
{\bf $\text{(iii)} \Leftrightarrow \text{(iv)}$.} We define
$$M = \{\text{maps $\xi \colon G \times \mathbb{Z} \to \mathbb{Z}[t^{\pm 1}]$} \mid \text{$\xi(s(g, i)) = t \xi(g,i)$ for any $g \in G, i \in \mathbb{Z}$}\}$$
and regard $M$ as a left $\mathbb{Z}[G \times \mathbb{Z}]$-module by $((g, i) \cdot \xi)(h,j) = \xi((h,j)(g,i))$. We claim that there exists a one-to-one correspondence between ${\rm Hom}_{\mathbb{Z}[t^{\pm 1}]}(H_1(X_{f \times \phi}, \tilde{x}), \mathbb{Z}[t^{\pm 1}])$ and the set of the lifts $\tilde{f} \colon \pi_1(X) \to M \rtimes (G \times \mathbb{Z})$ of $f$. For $\alpha \in {\rm Hom}_{\mathbb{Z}[t^{\pm 1}]}(H_1(X_{f \times \phi}, \tilde{x}), \mathbb{Z}[t^{\pm 1}])$, we define $\tilde{f}_\alpha \colon \pi_1(X) \to M \rtimes (G \times \mathbb{Z})$ by $\tilde{f}_\alpha(\gamma) = (\xi_{\alpha, \gamma}, (f\times\phi)(\gamma))$. Here, we denote the lift of $\gamma$ starting at $\tilde{x}_0$ by $\tilde{\gamma} \colon [0, 1] \to X_{f \times \phi}$ and then define $\xi_{\alpha, \gamma} \in M$ by $\xi_{\alpha, \gamma}(g,i) = \alpha((g,i) \cdot \tilde{\gamma})$, regarding $\tilde{\gamma}$ as representing a homology class of $H_1(X_{f \times \phi}, \tilde{x})$; since $\alpha$ is a homomorphism of $\mathbb{Z}[t^{\pm 1}]$-modules, $\xi_{\alpha, \gamma}$ satisfies the condition $\xi_{\alpha, \gamma}(s(g,i)) = t \xi_{\alpha, \gamma}(g,i)$. For $\gamma, \gamma' \in \pi_1(X_{f \times \phi})$, we have
$$\xi_{\alpha, \gamma\gamma'}(g,i) = \alpha((g,i) \cdot \tilde{\gamma}) + \alpha((g,i) \cdot (f \times \phi)(\gamma) \cdot \tilde{\gamma}')$$
and then find that $\tilde{f}_\alpha$ is a group homomorphism. Conversely, let $\tilde{f} \colon \pi_1(X) \to M \rtimes (G \times \mathbb{Z})$ be a lift of $f$. If $\tilde{f}(\gamma) = (\xi_{\gamma}, (f \times \phi)(\gamma))$ for $\gamma \in \pi_1(X)$, we define $\alpha_{\tilde{f}} \colon H_1(X_{f \times \phi}, \tilde{x}) \to \mathbb{Z}[t^{\pm 1}]$ by $\alpha_{\tilde{f}}((g,i) \cdot \tilde{\gamma}) = \xi_{\gamma}(g,i)$. Again, we can easily check that $\alpha_{\tilde{f}}$ is a well defined homomorphism of $\mathbb{Z}[t^{\pm 1}]$-modules, and that the correspondences $\alpha \mapsto \tilde{f}_\alpha$ and $\tilde{f} \mapsto \alpha_{\tilde{f}}$ are the inverses of each other.
\par In the notation of the previous paragraph, the lifts $(g,i) \cdot \tilde{\gamma}$ of $\gamma \in \pi_1(X)$ are loops if and only if $(f \times \phi)(\gamma) = (e, 0)$, and we should remark that $H_1(X_{f \times \phi})$ is generated by such elements $(g,i) \cdot \tilde{\gamma}$. Thus, the homomorphisms $\alpha \in {\rm Hom}_{\mathbb{Z}[t^{\pm 1}]}(H_1(X_{f \times \phi}, \tilde{x}), \mathbb{Z}[t^{\pm 1}])$ that vanish under the surjection to ${\rm Hom}_{\mathbb{Z}[t^{\pm 1}]}(H_1(X_{f \times \phi}), \mathbb{Z}[t^{\pm 1}])$ correspond to the lifts $\tilde{f} \colon \pi_1(X) \to M \rtimes (G \times \mathbb{Z})$ such that ${\rm Im}\; \tilde{f} \cap (M \times \{(e, 0)\}) = \{(0; e, 0)\}.$\\[5pt]
{\bf $\text{(iv)} \Leftrightarrow \text{(v)}$.} Let us see $M \rtimes (G \times \mathbb{Z}) \cong \mathbb{Z}[G \times \mathbb{Z}] \rtimes (G \times \mathbb{Z})$. In fact, $M \ni \xi \mapsto \sum_{g \in G} \xi(g)g^{-1} \in \mathbb{Z}[t^{\pm 1}][G] \cong \mathbb{Z}[G \times \mathbb{Z}]$ gives an isomorphism $M \cong \mathbb{Z}[G \times \mathbb{Z}]$ between the left $\mathbb{Z}[G \times \mathbb{Z}]$-modules, and hence we have $M \rtimes (G \times \mathbb{Z}) \cong \mathbb{Z}[G \times \mathbb{Z}] \rtimes (G \times \mathbb{Z})$.
\end{proof}

\begin{remark}
Let $A$ be the \textit{Alexander matrix} of a presentation of $\pi_1(X)$ associated to the representation $(\rho\circ f)\otimes\phi \colon \pi_1(X)\to \mathrm{GL}(|G|,\Z[t^{\pm1}])$, which is defined as the Jacobian matrix with respect to the free differential calculus. 
Then, Wada shows in \cite[Propositition 1]{Wada94-1} that there is a natural one-to-one correspondence between the kernel of $A$ and the set of \textit{derivations} of $\pi_1(X)$ with values in a $\Z[\pi_1(X)]$-module 
$\Z[t^{\pm1}]^{|G|}$. 
\end{remark}

%%%%%%%%%%%%%%%%%%%%%%%%%%%%%%%%%%%%%%%%%%%%%%%%%%%%%%%
\subsection{Proof of Theorem \ref{char-thm}}
Let us consider the case of a knot. By Theorem \ref{lifting-thm}, a finite group $G$ is a TAV group if and only if there exist a knot $K$ and an epimorphism $f \colon G(K) \to G$ such that $f \times \phi$ admits a nontrivial lift. We first see that we may assume some additional conditions on the image of the lift (Lemma \ref{m0-lem}), and then show that the existence of such a special lift is equivalent to the existence of a certain ideal of $\mathbb{Z}[\tilde{G}]$, where $\tilde{G}$ is a subgroup of $G \times \mathbb{Z}$ (Proposition \ref{ideal-lem}). By an algebraic argument on group theory, we find that there exists such an ideal if and only if $[G, G]$ is not a $p$-group (Lemmas \ref{pgp-lem}, \ref{cyclic-lem}, and \ref{metabel-lem}). This is an outline of the proof.
\par Let $K$ be a knot and $m \in G(K)$ a meridian. We assume that a homomorphism $f \colon G(K) \to G$ to a finite group $G$ is surjective and that $\phi(m) = 1$. Let $\tilde{G} \subset G \times \mathbb{Z}$ denote the image of $f \times \phi$, i.e., the subgroup generated by $[G, G] \times \{0\}$ and $(f(m), 1)$. We take a complete representative set $E \subset G \times \mathbb{Z}$ of $\tilde{G} \backslash (G \times \mathbb{Z})$. Since $\mathbb{Z}[G \times \mathbb{Z}] = \bigoplus_{x \in E} \mathbb{Z}[\tilde{G}] x \cong \mathbb{Z}[\tilde{G}]^{|E|}$ as a left $\mathbb{Z}[\tilde{G}]$-module, the group $\mathbb{Z}[G \times \mathbb{Z}] \rtimes \tilde{G} \subset \mathbb{Z}[G \times \mathbb{Z}] \rtimes (G \times \mathbb{Z})$ is isomorphic to $\mathbb{Z}[\tilde{G}]^{|E|} \rtimes \tilde{G}$. Thus, the existence of a nontrivial lift of $f \times \phi \colon G(K) \to \tilde{G}$ to $\mathbb{Z}[G \times \mathbb{Z}] \rtimes \tilde{G}$ is equivalent to that of a nontrivial lift to $\mathbb{Z}[\tilde{G}] \rtimes \tilde{G}$.

\begin{lemma}\label{m0-lem}
If $f \times \phi$ admits a nontrivial lift to $\mathbb{Z}[\tilde{G}] \rtimes \tilde{G}$, there exists a nontrivial lift $\tilde{f} \colon G(K) \to \mathbb{Z}[\tilde{G}] \rtimes \tilde{G}$ such that $\tilde{f}(m) = (0, (f \times \phi)(m))$.
\end{lemma}

\begin{proof}
Let $n$ be the order of $f(m) \in G$ and denote $(f \times \phi)(m)$ by $\mu \in \tilde{G}$. Since $\mu^n - 1$ is central in $\mathbb{Z}[\tilde{G}]$ and is not a zero-divisor, the map $\iota \colon \mathbb{Z}[\tilde{G}] \rtimes \tilde{G} \to \mathbb{Z}[\tilde{G}] \rtimes \tilde{G}$ defined by $\iota(\eta; g, i) = ((\mu^n - 1)\eta; g, i)$ for $\eta\in\Z[\tilde{G}],(g, i) \in \tilde{G}$ is an injective group homomorphism.
\par Let $\tilde{f}_0 \colon G(K) \to \mathbb{Z}[\tilde{G}] \rtimes \tilde{G}$ be a nontrivial lift. If $\tilde{f}_0(m) = (\eta, \mu)$, we set $\eta' = \sum_{j=0}^{n-1} \mu^j \eta \in \mathbb{Z}[\tilde{G}]$ and define $\tilde{f} \colon G(K) \to \mathbb{Z}[\tilde{G}] \rtimes \tilde{G}$ by
$$\tilde{f}(\gamma) = (\eta'; e, 0) \cdot (\iota \circ \tilde{f}_0)(\gamma) \cdot (\eta'; e, 0)^{-1}.$$
As $\iota$ is injective, $\tilde{f}$ is a nontrivial lift of $f \times \phi$. Furthermore, we have
$$\tilde{f}(m) = ((\mu^n - 1)\eta + (1 - \mu)\eta', \mu) = ((\mu^n - 1)\eta + (1 - \mu^n)\eta, \mu) = (0, \mu),$$
as required.
\end{proof}

\par Let $G$ be a finite group normally generated by a single element $g_0 \in G$. To see whether there exist a knot $K$ and an epimorphism $G(K) \to G$ that takes a meridian to $g_0$ such that the twisted Alexander polynomial vanishes, we define $\tilde{G} \subset G \times \mathbb{Z}$ to be the subgroup generated by $[G, G] \times \{0\}$ and $\mu = (g_0, 1)$, and examine whether a nontrivial finitely generated subgroup $\overline{G} \subset \mathbb{Z}[\tilde{G}] \rtimes \tilde{G}$ that is normally generated by a single element exists, where ``nontrivial'' means that the intersection with $[G, G] \times \{0\}$ is not trivial.
\par By Lemma \ref{m0-lem}, we may assume that $\overline{G}$ contains and is normally generated by $(0,\mu)$. Furthermore, we may assume that $\{0\} \times \tilde{G} \subset \overline{G}$. In fact, an epimorphism onto $\tilde{G}$ always admits a trivial lift to $\{0\} \times \tilde{G} \subset \mathbb{Z}[\tilde{G}] \rtimes \tilde{G}$, and the connected sum of a knot with a nontrivial lift and one with a trivial lift yields a knot with a nontrivial lift, whose image contains $\{0\} \times \tilde{G}$.
\par For a group $H$, let $I_H \subset \mathbb{Z}[H]$ denote the \textit{augmentation ideal}, i.e., the ideal spanned by the elements $h - 1$ for $h \in H$. Assuming $\overline{G} \supset \{0\} \times \tilde{G}$, we find that there exists a left ideal $I$ of $\mathbb{Z}[\tilde{G}]$ such that $\overline{G} = I \rtimes \tilde{G}$. By an elementary calculation, we also find that $\overline{G}$ is finitely generated if and only if $I$ is finitely generated as an ideal, and that the normal generation by $(0, \mu)$ is equivalent to the condition $I_{\tilde{G}}I = I$. To summarize,

\begin{proposition}\label{ideal-lem}
There exist a knot $K$ and an epimorphism $f \colon G(K) \to G$ such that $f$ takes a meridional loop to $g_0$ and $\Delta_K^{\rho \circ f}(t) = 0$ if and only if there exists a finitely generated nonzero left ideal $I$ of $\mathbb{Z}[\tilde{G}]$ such that $I_{\tilde{G}} I = I$.
\end{proposition} 

\par Let $G_0$ denote the subgroup $[G, G] \times \{0\}$ of $\tilde{G}$, which is equal to the commutator subgroup of $\tilde{G}$ and is isomorphic to $[G, G]$. For a subgroup $H$ of $G_0$, let $\tilde{H}$ denote the subgroup of $\tilde{G}$ generated by $H$ and $\mu$.

\begin{lemma}\label{pgp-lem}
Assume that $[G, G]$ is a $p$-group. If a left ideal $I$ of $\mathbb{Z}[\tilde{G}]$ satisfies $I_{\tilde{G}}I = I$, then $I = 0$.
\end{lemma}

\begin{proof}
Let $J$ denote the ideal $\mathbb{Z}[\tilde{G}]I_{G_0}$. Since $G_0$ is a normal subgroup of $\tilde{G}$, we have $J^n = \mathbb{Z}[\tilde{G}] I_{G_0}^n$. As $G_0$ is a finite $p$-group, $\bigcap_{i = 0}^\infty I_{G_0}^i = 0 $ by \cite[Theorem B]{gru} and hence $\bigcap_{i=0}^\infty J^i = 0$.
\par Suppose that there exists a left ideal $I \neq 0$ of $\mathbb{Z}[\tilde{G}]$ such that $I_{\tilde{G}}I = I$. Since $I \neq 0$, there exists $n \geq 0$ such that $I \subset J^n$ but $I \not\subset J^{n+1}$. Let $I'$ denote the left $\mathbb{Z}[\tilde{G}]$-module $(I + J^{n+1})/J^{n+1}$, which is a submodule of $J^n/J^{n+1}$. We should remark that $J^n/J^{n+1}$ is isomorphic to $\mathbb{Z}[\tilde{G}/G_0] \otimes_{\mathbb{Z}} (I_{G_0}^n/I_{G_0}^{n+1})$ as a module of $\mathbb{Z}[\tilde{G}]/J \cong \mathbb{Z}[\tilde{G}/G_0]$ and that $I_{\tilde{G}}/J = (I_{\langle \mu \rangle} + J)/J \cong I_{\tilde{G}/G_0}$. Since $\tilde{G}/G_0$ is an infinite cyclic group, we can identify $\mathbb{Z}[\tilde{G}/G_0]$ with $\mathbb{Z}[t^{\pm 1}]$ and then $I_{\tilde{G}}^i J^n/J^{n+1} \cong (t-1)^i \mathbb{Z}[t^{\pm 1}] \otimes_{\mathbb{Z}} (I_{G_0}^n/I_{G_0}^{n+1})$. Thus, we have $\bigcap_{i=0}^\infty I_{\tilde{G}}^i I' \subset \bigcap_{i=0}^\infty I_{\tilde{G}}^iJ^n/J^{n+1} = 0$, which implies that $\bigcap_{i=0}^\infty I_{\tilde{G}}^i I \subset J^{n+1}$; this is a contradiction.
\end{proof}

\par Suppose that $[G, G]$ is not a $p$-group and let $H \subset G_0$ be a minimal subgroup not having prime-power order. By \cite[Lemma 1]{gru}, $H$ is

\begin{itemize}
\item[(i)] a cyclic group of order $pq$ for some distinct primes $p, q$, or
\item[(ii)] a non-abelian group of order $p^n q$ for distinct primes $p,q$ and a positive integer $n$, where the Sylow $p$-subgroup $P$ of $H$ is a minimal normal subgroup while the Sylow $q$-subgroups are maximal.
\end{itemize}

\noindent We remark that in the case of (ii) the subgroup $P$ is an abelian group isomorphic to $\mathbb{Z}_p^n$.

\begin{lemma}\label{cyclic-lem}
In the case of {\rm(i)}, there exists a finitely generated nonzero ideal $I$ of $\mathbb{Z}[\tilde{H}]$ such that $I_{\tilde{H}} I = I$.
\end{lemma}

\begin{proof}
Let $h$ be a generator of $H$. As shown in \cite{gru}, the intersection $\bigcap_{i=0}^\infty I_H^i$ is equal to the ideal $J$ of $\mathbb{Z}[H]$ generated by $\varphi_1(h) \varphi_p(h) \varphi_q(h)$, where $\varphi_d(h)$ is the $d$-th cyclotomic polynomial, and the multiplication of $h - 1$ is invertible in $J$; in particular, we have $I_H J = J$. The ideal $I = \mathbb{Z}[\tilde{H}] J$ of $\mathbb{Z}[\tilde{H}]$ is clearly finitely generated and satisfies
$$I_{\tilde{H}} I \supset \mathbb{Z}[\tilde{H}] I_H J = \mathbb{Z}[\tilde{H}] J = I,$$
i.e., $I_{\tilde{H}} I = I$.
\end{proof}

\begin{lemma}\label{metabel-lem}
In the case of {\rm(ii)}, there exists a finitely generated nonzero ideal $I$ of $\mathbb{Z}[\tilde{H}]$ such that $I_{\tilde{H}} I = I$.
\end{lemma}

\begin{proof}
Gruenberg \cite[Theorem B]{gru} shows that $I_P \subset \bigcap_{i=0}^\infty I_H^i$, and in the proof it is stated that $\mathbb{Z}[H] I_P = \bigcap_{i=0}^\infty I_H^i$; in fact, we can verify $\mathbb{Z}[H] I_P \supset \bigcap_{i=0}^\infty I_H^i$ by projecting both sides to $\mathbb{Z}[H/P]$.
\par Let $J$ be the ideal $I_H I_P$ of $\mathbb{Z}[H]$. We claim that $I_H J = J$. To see this, we remark that $I_H^k I_P \supset (\bigcap_{i=0}^\infty I_H^i) I_P = \mathbb{Z}[H] I_P^2$ for any $k$. We can regard $I_H^k I_P/\mathbb{Z}[H]I_P^2$ as a left module of $\mathbb{Z}[H]/\mathbb{Z}[H]I_P \cong \mathbb{Z}[H/P] \cong \mathbb{Z}[x]/(x^q -1)$. Since $I_P/ I_P^2 \cong P \cong \mathbb{Z}_p^n$,
$$I_H^k I_P/\mathbb{Z}[H]I_P^2 \cong (x-1)^k\mathbb{Z}[x]/(x^q - 1) \otimes \mathbb{Z}_p^n \cong ((x-1)^k\mathbb{Z}_p[x]/(x^q-1))^n.$$
There is an injective homomorphism $\mathbb{Z}_p[x]/(x^q-1) \to \mathbb{Z}_p[x]/(x-1) \oplus \mathbb{Z}_p[x]/(\varphi_q(x))$ and $x-1$ is zero in $\mathbb{Z}_p[x]/(x-1)$ and invertible in $\mathbb{Z}_p[x]/(\varphi_q(x))$. Thus, we find $I_H I_P/\mathbb{Z}[H]I_P^2 = I_H^2 I_P/\mathbb{Z}[H]I_P^2$, i.e., $I_H J = J$ as claimed.
\par As in the proof of Lemma \ref{cyclic-lem}, $I = \mathbb{Z}[\tilde{H}]J$ is a finitely generated nonzero ideal of $\mathbb{Z}[\tilde{H}]$ such that $I_{\tilde{H}} I = I$.
\end{proof}

\begin{proof}[Proof of Theorem \ref{char-thm}]
Let $g_0 \in G$ be any element normally generating $G$ and define $\tilde{G}$ as above. By Proposition \ref{ideal-lem}, it is sufficient to determine when there exists a finitely generated nonzero left ideal $I$ of $\mathbb{Z}[\tilde{G}]$ such that $I_{\tilde{G}} I = I$. If $[G, G]$ is a $p$-group, Lemma \ref{pgp-lem} shows the non-existence of such an ideal, and hence $\Delta_K^{\rho \circ f}(t)$ does not vanish for any $K$ and $f$. If $[G, G]$ is not a $p$-group, let $H \subset G_0$ be a minimal subgroup not having prime-power order. There are two cases (i), (ii) as above, but in either case there exists a finitely generated nonzero left ideal $J$ of $\mathbb{Z}[\tilde{H}]$ such that $I_{\tilde{H}} J = J$ by Lemmas \ref{cyclic-lem} and \ref{metabel-lem}. Define $I = \mathbb{Z}[\tilde{G}] J$. Since $I_{\tilde{G}} \supset I_{\tilde{H}}$, we find $I_{\tilde{G}} I = I$ and hence there exist $K$ and $f$ such that $\Delta_K^{\rho \circ f} (t) = 0$.
\end{proof}

%%%%%%%%%%%%%%%%%%%%%%%%%%%%%%%%%%%%%%%%%%%%%%%%%%%%%%%%%%%%%%%%%%%%%%%%%%%%%%%%%%%%%%%%%%%%%%%%%%%%%%%%%%%%%%%%%%%%
\section{Concluding remark}\label{sec:remark}

In this section, we exhibit the upper bounds of the TAV order of several knots explicitly, and discuss related problems on the TAV order and twisted Alexander polynomials of $3$-dimensional manifolds associated to representations 
of finite groups. 

By computer-aided calculation, we can provide 
the upper bound of $\ord(K)$ for several non-fibered knots as follows:

$$
126\leq 
\ord(K)\leq
\begin{cases}
168 & \text{if}~K= 7_4, 8_3,\\
336 & \text{if}~K= 5_2, 6_1, 9_5, 9_{10}, 10_3, 10_{16}, 10_{33}, 10_{65}, 10_{74},10_{122},\\
576 & \text{if}~K= 9_7, 10_1, 10_7, 10_{77},\\
2520 & \text{if}~K= 7_2, 8_1, 8_6, 8_{14}, 9_4, 9_{13}, 10_4, 10_{31}, 10_{35}, 10_{163},\\
5040 & \text{if}~K= 8_8,\\
20160 & \text{if}~K= 9_{14}, 9_{19}, 9_{23}, 9_{38}, 10_{24}, 10_{30}, 10_{38}, 10_{68}, 10_{97}, 10_{129},\\
40320 & \text{if}~K=10_8,10_{14}, 10_{34}, 10_{144}.
\end{cases}
$$

If a knot $K$ admits an epimorphism $\pi \colon G(K)\to G(K')$ where 
$K'$ is one of the above non-fibered knots, 
then $\ord(K)$ has the same upper bound with $\ord(K')$ by Theorem~\ref{thm:main-1}(iv). 
As a first problem, we propose the following:  

\begin{problem}
Determine the TAV order $\ord(K)$ of the above $41$ non-fibered knots $K$.  Moreover, what is the TAV order of the remainder $79$ non-fibered prime knots with $10$ or fewer crossings? 
\end{problem}

As for the inequality in Theorem \ref{thm:main-1}(v), 
there is an example such that the equality does not hold. 
In fact, for the periodic knot $K=10_{120}$ of order $2$ and its 
quotient knot $K'=5_2$ (see \cite{KS08-1}), we see from Theorem~\ref{thm:main-2} 
that $\ord(10_{120})<\ord(5_2)$ holds. However, at this point, we do not know the existence of a proper degree one map $E_K\to E_{K'}$ such that $\ord(K)$ is strictly smaller than $\ord(K')$.  

\begin{problem}
Find a proper degree one map $E_K\to E_{K'}$ such that $\ord(K)<\ord(K')$. Moreover, is there a pair of distinct non-fibered knots $K,K'$ such that the equalities in Theorem \ref{thm:main-1}(v) and (vi) hold?
\end{problem}

In view of Theorem \ref{thm:main-4}, for any TAV group $G$, namely, any finite group $G$ normally generated by a single element, and its commutator subgroup $[G,G]$ is not a $p$-group, there exist a non-fibered knot $K$ and an epimorphism $f \colon G(K) \to G$ such that the corresponding twisted Alexander polynomial $\D_{K}^{\rho\circ f}(t)$ is zero. For example, the dihedral group $D_{15}$ is a TAV group. 

\begin{problem}
Find a non-fibered knot $K$ and an epimorphism $f \colon G(K) \to D_{15}$ such that $\D_K^{\rho\circ f}(t)=0$. 
\end{problem}

%Using Theorem \ref{thm:main-4}, 
Furthermore, 
we know that the intersection of the image of the TAV order $\ord|_{\mathcal{N}} \colon \mathcal{N}\to\mathbb{N}$ and the closed interval $[1,126]$ is contained in the finite set 
$$
\{24,30,42,48,60,66,70,72,78,84,90,96,102,110,114,120,126\},
$$ 
%Furthermore, 
and that Theorem \ref{thm:main-4} guarantees the existence of TAV groups of these orders. 
However, it is not known whether these values will actually be realized as the TAV orders of 
%suitable 
non-fibered knots. At least we see from Theorem \ref{thm:main-2} that $\mathrm{Im}\, \ord$ contains the finite set $\{24,60,96,120\}$.

\begin{problem}
Determine the image of the TAV order $\ord|_{\mathcal{N}} \colon \mathcal{N}\to\mathbb{N}$. 
\end{problem}

Finally, 
we mention a related problem on representations of $3$-manifold groups. 
For a compact, orientable, connected $3$-manifold 
$N$ with toroidal or empty boundary, if $\phi\in H^1(N;\Z)=\mathrm{Hom}(\pi_1(N),\Z)$ is a non-fibered class, then the \textit{twisted Alexander vanishing (TAV) order} $\ord(N,\phi)$ is defined to be the smallest order of a finite group $G$ such that there exists an epimorphism $f \colon \pi_1(N)\to G$ 
with $\D_{N,\phi}^{\rho\circ f}(t)=0$ (see \cite{MS22-1}). 
However, nothing is known for $3$-manifold groups so far. 
Accordingly, we conclude the paper with the following problem: 

\begin{problem}
%Find a finite group $G$ and an epimorphism $f \colon \pi_1(N)\to G$ such that $\D_{N,\phi}^{\rho\circ f}(t)=0$, and show 
Study the basic properties of $\ord(N,\phi)$ as described in Theorem~\ref{thm:main-1}. 
%Moreover, provide an explicit formula for $\D_{N,\phi}^{\rho\circ f}(t)$ in the non-vanishing case. 
\end{problem}

%%%%%%%%%%%%%%%%%%%%%%%%%%%%%%%%%%%%%%%
\subsection*{Acknowledgments}
The authors would like to thank the anonymous referee for the evaluation of our paper and for useful suggestions. They are supported in part by 
JSPS KAKENHI Grant Numbers JP16K05159, JP20K03596, JP20K14309, and JP21K03253. A part of this research was done during the second and third authors' stays at the Research Institute for Mathematical Sciences, Kyoto University. They would like to express their sincere thanks for the hospitality. 

%%%%%%%%%%%%%%%%%%%%%%%%%%%%%%%%%%%%%%%
\bibliographystyle{amsplain}

\end{document}